\documentclass[a4paper,12pt]{amsart}
\usepackage[margin=3cm]{geometry}
\usepackage[utf8]{inputenc}
\usepackage{lmodern}
\usepackage{url}
\usepackage{amsmath, amsthm, amssymb, amsfonts, amscd, accents, bm}
\usepackage{mathtools}
\mathtoolsset{showonlyrefs}  
\usepackage{bbm}
\usepackage{todonotes}
\usepackage{mdwlist}
\usepackage{paralist}
\usepackage{graphicx}
\usepackage{epstopdf}
\usepackage{datetime}
\usepackage{afterpage}
\usepackage{hyperref}
\usepackage{caption}
\usepackage{subcaption}
\usepackage{dsfont}
\usepackage[foot]{amsaddr}

\usepackage{mathtools}

\newcommand\rurl[1]{%
  \href{https://#1}{\nolinkurl{#1}}%
}

\makeatletter
\def\@makefnmark{%
  \leavevmode
  \raise.9ex\hbox{\fontsize\sf@size\z@\normalfont\tiny\@thefnmark}}

\makeatletter
\def\bign#1{\mathclose{\hbox{$\left#1\vbox to8.5\p@{}\right.\n@space$}}\mathopen{}}

\usepackage{eqparbox}
\usepackage{makebox}
 
\DeclareMathOperator{\argmin}{argmin}

\theoremstyle{definition}

\newtheorem{definition}{Definition}[section]

\theoremstyle{plain}
\newtheorem{theorem}[definition]{Theorem}
\newtheorem{lemma}[definition]{Lemma}
\newtheorem{corollary}[definition]{Corollary}
\newtheorem{proposition}[definition]{Proposition}

\newtheorem{remark}[definition]{Remark}

\newtheorem{conj}[definition]{Conjecture}
\numberwithin{equation}{section}

\def\R{{\mathbb R}}

\def\C{{\mathbb C}}

\def\K{{\mathbb K}}

\def\V{{\mathcal V}}
\def\D{{\mathbb D}}

\DeclareMathOperator{\suppp}{supp}

\DeclareMathOperator{\sech}{sech}
\DeclareMathOperator{\csch}{csch}

\title{Stable STFT Phase retrieval and Poincar\'e inequalities}
\author{Martin Rathmair}

\address{Research Network DataScience@UniVie, University of Vienna, Kolingasse 14-16, 1090 Vienna, Austria}
\email{martin.rathmair@univie.ac.at}
\begin{document}

\begin{abstract}
    In recent work [P. Grohs and M. Rathmair.  Stable Gabor Phase Retrieval and Spectral Clustering. Communications on Pure and Applied Mathematics (2018)]
    and 
    [P. Grohs and M. Rathmair. Stable Gabor phase retrieval for multivariate functions. Journal of the European Mathematical Society (2021)] 
    the instabilities of Gabor phase retrieval problem, i.e. reconstructing $ f\in L^2(\R)$ from its spectrogram $|\V_g f|$ where
    $$
    \V_g f(x,\xi) = \int_{\R} f(t)\overline{g(t-x)}e^{-2\pi i \xi t}\,\mbox{d}t,
    $$
    have been classified in terms of the connectivity of the measurements.
    These findings were however crucially restricted to the case where the window $g(t)=e^{-\pi t^2}$ is Gaussian. 
    In this work we establish a corresponding result for a number of other window functions including the one-sided exponential $g(t)=e^{-t}\mathds{1}_{[0,\infty)}(t)$  and  $g(t)=\exp(t-e^t)$. 
    As a by-product we establish a modified version of Poincar\'e's inequality which can be applied to non-differentiable functions and may be of independent interest.
\end{abstract}

\maketitle

\section{Introduction}

Phase retrieval refers to the problem of determining a function given its modulus only.
Such kind of questions are encountered in a number of applications including 
diffraction imaging \cite{miao:beyondcrystallography}, quantum mechanics \cite{CORBETT200653}, audio \cite{waldspurger:waveletpr} and astronomy \cite{dainty:astronomy}.\\
More precisely, given $\Omega$ a set and $V\subseteq \K^\Omega$ a function class with $\K\in\{\R,\C\}$, phase retrieval asks to
\begin{center}
    reconstruct $F\in V$, given its pointwise modulus $|F|$.
\end{center}
Typically $V$ is a linear space. If $|\lambda|=1$, $F$ and $\lambda F$ can not be distinguished from the phaseless observations. Thus reconstruction is at best possible up to a constant phase factor. We say that $V$ does phase retrieval if 
$$
F,H\in V:\, |F|=|H| \quad \Rightarrow \quad \exists \,|\lambda|=1:\,\, H=\lambda F.
$$

As for any inverse problem, stability is a property of fundamental importance: In order for the reconstruction task to be a well-posed question one requires that similar measurements can only arise from similar inputs. More formally, this can be expressed in terms of Lipschitz stability inequality:
Fix $\|\cdot\|$ a norm. We say that $V$ does stable phase retrieval if there exists $C>0$ such that
$$
\forall F,H\in V:\quad \inf_{|\lambda|=1} \|H-\lambda F\| \le C \||H|-|F|\|.
$$
If $V$ does phase retrieval and $\dim V<\infty$, then automatically $V$ does stable phase retrieval. 
In infinite dimensions however, the situation becomes more interesting and there exist results in two different directions.
On the one hand, infinite-dimensionsional spaces $V$ that satisfy certain rather general conditions do not do stable phase retrieval \cite{cahill:infdimhilbert,alaifari:banach}.
On the other hand, there are a couple of very recent constructions of infinite-dimensional spaces which do stable phase retrieval \cite{calderbank2022stable,christ2023examples,freeman:stablefctspace}.\\
In order to circumvent the issue of a lack of Lipschitz stability,
one may consider a more flexible concept by allowing the constant $C$ to depend on $F$, i.e., 
$$
\forall H\in V:\quad \inf_{|\lambda|=1} \|H-\lambda F\| \le C(F) \||H|-|F|\|.
$$
We point out that even deciding whether $C(F)<\infty$ is a highly non-trivial matter.
It is the subject of the present article to study precisely this type of concept of local Lipschitz stability and investigate its behaviour.

\subsection{Related results}
For a large class of phase retrieval problems instabilities can be built by employing 'multi-component' constructions, cf. \cite{cahill:infdimhilbert, alaifari:banach}. From an abstract point of view the idea is to pick a pair of functions $F_1,F_2\in V$ that essentially live on disjoint domains, i.e. their product $F_1F_2$ is small. Then, adding two such components results in an instability since
$$
|F_1 + F_2| \approx |F_1| + |F_2| \approx |F_1-F_2|.
$$
The availability of such disjoint components is of course dictated by the concrete choice of the space $V$. 
In many cases, a common feature of multi-component instabilities is disconnectedness.
To make matters more concrete we introduce the following isoperimetric quantity.
\begin{definition}
    Let $\Omega\subseteq \R^d$ be a domain and let $w\in C(\Omega,[0,\infty))\cap L^1(\Omega)$ be a weight.
    The \emph{Cheeger constant} of $w$ on $\Omega$ is defined by 
    \begin{equation}
    h(w,\Omega) \coloneqq \inf_E \frac{\int_{\partial E \cap \Omega} w(s)\,\mbox{d}s}{\min\left\{\int_E w(x)\,\mbox{d}x, \,\int_{\Omega\setminus E} w(x)\,\mbox{d}x \right\}},
    \end{equation}
    where $E$ runs through the set of all subsets of $\Omega$ with smooth interior boundary $\partial E\cap \Omega$.
\end{definition}
The Cheeger constant $h(w,\Omega)$ quantifies the connectedness of $w$ on $\Omega$: A small Cheeger constant indicates that there exists a partition into $E$ and $\Omega\setminus E$ such that $w$ has roughly the same $L^1$ volume while the weight is rather small on the seperating boundary; this is the disconnected case. In the connected case however (to have a concrete example, think of a Gaussian)  both of these objectives are not feasible at the same time.
We refer to the introductory section of \cite{grohs:gaborspectral} for more details on phase retrieval instabilities, the role of connectivity for phase retrieval as well as the Cheeger constant.\\ 

To provide some more context let us briefly discuss our results in \cite{grohs:gaborspectral, grohs:gabormultivariate}. 
The central object in both of these articles is the short-time Fourier transform.
\begin{definition}
Let $f,g\in L^2(\R)$.
    The \emph{short-time Fourier transform} (STFT) of $f$ w.r.t. the window function $g$ is defined as
    \begin{equation}
        \V_g f(x,\xi)\coloneqq \int_{\R} f(t)\overline{g(t-x)}e^{-2\pi i \xi t}\,\mbox{d}t,\quad x,\xi\in\R.
    \end{equation}
\end{definition}
The results in \cite{alaifari:banach} show -- irrespectively of the choice of $g$ -- that $V=\V_g(L^2(\R))$ allows for construction of multi-component instabilities.
The main insight of our earlier result is that disconnectedness of the spectrogram is the only source of instabilities for the STFT phase retrieval problem with Gaussian window function.
\begin{theorem}[\cite{grohs:gabormultivariate}, informal]\label{thm:maingabor}
    Let $\varphi(t)=e^{-\pi t^2}$ and let $\Omega\subseteq \R^2$ a domain. 
    There exists a universal constant $c>0$ such that for all $F,H\in \V_\varphi(L^2(\R))$\footnote{We identify $F,H$ with their respective restrictions to the subdomain $\Omega$.}
    \begin{equation}
        \inf_{|\lambda|=1} \|H-\lambda F\|_{L^1(\Omega)} \le c \big(1+ h(|F|, \Omega)^{-1} \big) \||H|-|F|\|_{W(\Omega)},
    \end{equation}
with $\|\cdot\|_{W(\Omega)}$ a certain first order weighted Sobolev norm.
\end{theorem}
\subsection{Aim of this work}
Theorem \ref{thm:maingabor} heavily relies on the fact that $\V_\varphi(L^2(\R))$ consists -- up to a simple manipulation -- of entire functions exclusively. 
Once the Gaussian $\varphi$ is replaced by any other window function this strong structural property goes missing and the proof method breaks down.
It is the purpose of this present work to establish the connection between local Lipschitz stability and the connectivity of the spectrogram in a more general setting and to show that this is actually not a consequence of the fact that the functions under consideration are holomorphic.

\subsection{Main results}
Before we present the main results of this paper we need to settle quite some terminology.
We will work with a novel concept of stability based on the reformulation of phase retrieval in the lifted setting. That is,
\begin{center}
    reconstruct $F\otimes\overline{F}$, given $|F|^2$.
\end{center}
Note that $|F|^2$ is just the restriction of $F\otimes\overline{F}$ to the diagonal $D=\{(z,z), z\in\R^d\}\subseteq \R^{2d}$. The task may therefore be understood as an extension problem from $D$ to all of $\R^{2d}$.
From this point of view the natural quantitative concept looks as follows.
\begin{definition}\label{def:prstable}
    Let $V\subseteq L^4(\R^d)$ be a linear space and let $\rho:\R^{2d}\to [0,\infty)$ be a measurable weight.
    We say that $V$ does $\rho$-stable lifted phase retrieval ($\rho$-SLPR) if there exists a finite constant $C>0$ such that 
    \begin{equation}
        \forall F,H\in V:\quad \left\|F\otimes\overline{F}-H\otimes\overline{H}\right\|_{L^2(\rho)} \le C \left\||F|^2-|H|^2\right\|_{L^2}.
    \end{equation}
\end{definition}
A central role will be played by Poincar\'e inequalities. 
For the sake of maximal flexibility we work with a non-symmetric version. 
\begin{definition}
    Let $v\in L^1(\R^d)$ be non-negative and let $w:\R^d\to [0,\infty)$ be measurable.
    We say that $v,w$ support a \emph{Poincar\'e inequality} if there exists a finite number $C>0$ such that 
    \begin{equation}
        \forall u\in C_b^\infty(\R^d):\quad \inf_{c\in\C} \|u-c\|_{L^2(v)}^2 \le C \|\nabla u\|_{L^2(w)}^2.
    \end{equation}
    The smallest possible constant $C$ will be denoted by $C_P(v,w)$ and is called the \emph{Poincar\'e constant}\footnote{this number may be finite or infinite}. 
    If $v=w$ we just write $C_P(w)$ instead of $C_P(v,w)$.
\end{definition}
The philosophy is as follows: The ideal case is that $w$ supports a Poincar\'e inequality. However, sometimes this is not the case (for example due to insufficient decay of $w$) and one needs to make amends. Weakening the requirement by replacing the weight $w$ on the left hand side by $v\le w$ potentially changes things in the sense that $C_P(v,w)<\infty$ while $C_P(w)=\infty$, cf. Section \ref{subsec:estspoincconst}.\\

A key tool that will appear at multiple places in the proofs is replacing weights by translates of themselves.
In order to have the correct control on the error made in these steps we introduce the following concept.
\begin{definition}
     We say that $\gamma\in C(\R^d,[0,\infty))$ is a \emph{translation stable weight} (TSW)   if there exists a locally bounded function $\mu:\R^d\to \R_+$ such that 
    \begin{equation}
        \forall z,\tau\in\R^d:\quad \gamma(z+\tau)\le \mu(\tau) \gamma(z).
    \end{equation}
\end{definition}
Prototypes of TSWs are $\gamma(t)=e^{-a|t|}$ and $\gamma(t)=(1+|t|)^{-a}$ where $a>0$, as well as products and tensor products of these types of functions.
Note that Gaussians are not TSWs.
\begin{definition}
    Let $\gamma\in L^1(\R^d)$ be a non-negative weight on $\R^d$.
    Associated to such a $\gamma$ is another weight $\Gamma$ on the product space $\R^d\times\R^d$ defined by 
    \begin{equation}
        \Gamma(z,z')=\Gamma[\gamma](z,z')\coloneqq \left(\gamma \ast \mathcal{R}\gamma \right)(z-z'), \quad z,z'\in\R^d,
    \end{equation} 
    where $\mathcal{R}$ denotes the reflection operator, i.e. $\mathcal{R}f(x)=f(-x)$.
\end{definition}
Last but not least, we need to settle w.r.t. which distance notions we measure the respective errors in our stability analysis.
\begin{definition}
    Given $\chi:\R^d\to [0,\infty)$ measurable we define a weighted mixed Lebesgue norm by
    \begin{equation}
        \|F\|_{\mathcal{L}(\chi)}\coloneqq \left\| \tau \mapsto \left\| F  \right\|_{L^2(B_1(\tau))} \right\|_{L^4(\chi)}.
    \end{equation}
    If $\chi=1$, we will write $\|\cdot\|_{\mathcal{L}}$ instead of $\|\cdot\|_{\mathcal{L}(\chi)}$.
    Moreover, we define 
    $$
    \mathfrak{d}(\phi,\psi):= \| \phi^2-\psi^2\|_{L^2(\R^d)}^{1/2}, \quad \phi,\psi\in L^4(\R^d).
    $$
\end{definition}
The following properties are easily verified:
\begin{itemize}
    \item $(L^4_+(\R^d),\mathfrak{d})$ forms a metric space with $L^4_+(\R^d)=\{\phi\in L^4(\R^d):\, \phi\ge 0\,\, \text{a.e.}\}$.
    \item $\mathfrak{d}$ is consistent w.r.t. scaling in the sense that 
$ \mathfrak{d}(c\phi,c\psi)=c\mathfrak{d}(\phi,\psi)$ if $c\ge0$.
\item If $\chi$ is bounded then  $(L^4(\R^d), \|\cdot\|_{\mathcal{L}(\chi)})$ forms a normed vector space (this follows from observing that 
$\|F\|_{\mathcal{L}}^4 = \| |F|^2 \ast \mathds{1}_{B_1} \|_2^2$ and Young's inequality).
\end{itemize}
Our first main result takes a rather general form. Its proof is given in Section \ref{ref:proofabstractstability}.
\begin{theorem}\label{thm:abstractstability}
    Let $\chi:\R^d\to [0,1]$ be a measurable weight and suppose that
    \begin{itemize}
        \item $\gamma$ is an integrable TSW and
        \item $V\subseteq L^4(\R^d)$ does $\Gamma$-SLPR, where $\Gamma=\Gamma[\gamma]$.
    \end{itemize}
    There exists $C>0$ (depending on $d$, $V$ and $\gamma$ only) such that   
    \begin{equation}\label{eq:mainabstract}
    \forall F,H\in V:\quad 
        \inf_{|\lambda|=1} \|H-\lambda F\|_{\mathcal{L}(\chi)} \le C(1+C_P(w\chi, w))^{\frac14}\cdot  \mathfrak{d}(|F|,|H|),
    \end{equation}
    with $w=(|F|^2\ast \gamma)^2$.
\end{theorem}

Neglecting technical details, the content of Theorem \ref{thm:abstractstability} can be summarised as follows:
If $V$ does stable lifted phase retrieval, then the local Lipschitz constant $C(F)$ of the associated phase retrieval problem is controlled in terms of the Poincar\'e constant w.r.t. the weight $w=(|F|^2\ast\gamma)^2$.
While the statement is pleasing as its rather general it contains the crucial assumption of $V$ doing $\Gamma$-SLPR.
To fill the theorem with life we will show that it leads to stability estimates for STFT phase retrieval.
We subsume the configurations which we can deal with under the following notion of admissibilty.
\begin{definition}\label{asspt:ggammapair}
A pair $(g,\gamma)$ consisting of a window function $g\in L^2(\R)$ and a TSW $\gamma$ on $\R^2$ is called an \emph{admissible pair} if 
\begin{itemize}
    \item the modulus of the ambiguity function $|\V_g g|$ is a TSW, i.e.,
    $$
    |\V_g g(z+\tau)|\le \mu(\tau) |\V_g g(z)|,\quad z,\tau\in\R^2
    $$
    with $\mu$ locally bounded and
    \item the controlling function $\mu$ satisfies the integrability condition $\mu \in L^2(\gamma\ast\mathcal{R}\gamma)$.
\end{itemize}
\end{definition}
In order to highlight that the above notion of admissibility is a meaningful concept we show that it is satisfied by two natural window functions collected in the lemma below. 
\begin{lemma}\label{lem:expadmpairs}
    Given $a,b>0$, let $\gamma(x,\xi)=\exp(-a|x|-b|\xi|)$.
    Assume either of the following configurations:
    \begin{enumerate}[(i)]
        \item Let $g(t)=\exp(t-e^t)$, $a>2\pi^2$ and $b>2$.
        \item Let $g(t)=e^{-t}\mathds{1}_{(0,\infty)}(t)$, $a>2$ and $b>0$.
    \end{enumerate}
    Then $(g,\gamma)$ forms an admissible pair.
\end{lemma}
\begin{figure}
    \centering
    \includegraphics[width=0.65\linewidth]{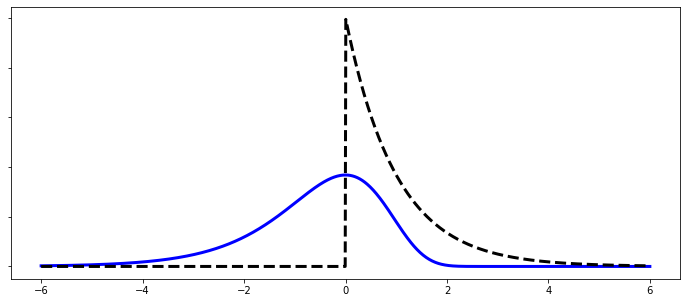}
    \caption{The two window functions from Lemma \ref{lem:expadmpairs}}
    \label{fig:enter-label}
\end{figure}

The proof of the lemma can be found in Section \ref{sec:prooflemmaadmpairs}. 

\begin{remark}[Further windows]
    The \emph{fundamental identity of time-frequency analysis} \cite[formula (3.10)]{grochenig:foundations} states that 
    $$
    \V_\phi \psi(x,\xi) = e^{-2\pi i x\xi} \cdot \V_{\hat\phi} \hat\psi(\xi,-x),\quad \phi,\psi\in L^2(\R).
    $$
    Therefore, with $\widetilde{\gamma}(x,\xi):= \gamma(\xi,-x)$ we have that $(g,\gamma)$ is an admissible pair if and only if $(\hat{g},\widetilde{\gamma})$ is an admissible pair.
    Hence, for every window $g$ appearing in the list in Lemma \ref{lem:expadmpairs} we could have included its Fourier transform $\hat{g}$ as well.
    More generally, one could also include the fractional Fourier transform of any such window by making use of the relation in \cite[Theorem 2.3]{liehr:founddiscbarr}.
    For the sake of compactness we refrained from explicitly including these as further examples in Lemma \ref{lem:expadmpairs}.
\end{remark}
Our second main result establishes a link between local Lipschitz stability for STFT phase retrieval and Poincar\'e constants and is proven in Section \ref{sec:stftstability}. The result reads as follows.
\begin{theorem}\label{thm:stabilitystft}
    Let $\chi: \R^2\to[0,1]$ be a measurable weight and let $(g,\gamma)$ be an admissible pair.
    There exists $C>0$ (depending on $g$ and $\gamma$ only) such that for all 
    $F,H\in \V_g(L^2(\R))$
    \begin{equation}
        \inf_{|\lambda|=1} \|H-\lambda F\|_{\mathcal{L}(\chi)} \le 
        C (1+C_P(w\chi,w))^{\frac14}\cdot \mathfrak{d}(|F|,|H|),
    \end{equation}
    with $w=(|F|^2\ast \gamma)^2$.
\end{theorem}

\begin{remark}[Comparison with earlier results]
In the default case $\chi\equiv 1$, the Poincar\'e constant appearing on the right is  the ordinary one, $C_P(w)$.
The celebrated \emph{Cheeger inequality} \cite{Cheeger:lowerbound} asserts that 
Poincar\'e and Cheeger constants are related by 
\begin{equation}
    C_P(w) \le \frac4{h(w,\R^d)^2}.
\end{equation}
Thus, the Poincar\'e constant appearing on the right in \eqref{eq:mainabstract} may be replaced by the squared reciprocal of the Cheeger constant.
Hence, Theorem \ref{thm:abstractstability} is in fact of the same flavour as Theorem \ref{thm:maingabor}.
A major improvement compared to the results in \cite{grohs:gaborspectral,grohs:gabormultivariate}
is that here the differences of the spectrograms is computed w.r.t. $\mathfrak{d}$ instead of the heavier weighted first order Sobolev norm in Theorem \ref{thm:maingabor}.\\
Another difference is the precise shape of the weight which goes into the Poincar\'e constant. Instead of taking the spectrogram as the weight, here it is a smoothed version of the spectrogram. Corollary \ref{cor:compactK} below suggests that this convolution might actually have a beneficial effect as it avoids issues arising from zeros of the weight function.
\end{remark}

As an immediate consequence of Theorem \ref{thm:stabilitystft} we get a statement on stable phase reconstruction on compact sets.
\begin{corollary}\label{cor:compactK}
    Let $\chi=\mathds{1}_K$ with $K\subseteq \R^2$ compact and let $(g,\gamma)$ be an admissible pair.
    Moreover, let $f\in L^2(\R)$ and $F=\V_g f$.
    There exists a finite constant $L>0$ (depending on $f,g,\gamma$ and $K$) such that 
    $$
    \forall H\in \V_g(L^2(\R)):\quad 
    \inf_{|\lambda|=1} \|H-\lambda F\|_{\mathcal{L}(\chi)}
    \le L\cdot  \mathfrak{d}(|F|,|H|).
    $$
\end{corollary}
As the proof is rather short we present it right here.
\begin{proof}
    If $f=0$, by Jensen's inequality
    $$
    \|H\|_{\mathcal{L}(\chi)}^4 \le \|H\|_{\mathcal{L}}^4 = 
    \| |H|^2 \ast \mathds{1}_{B_1(0)}\|_{L^2}^2 
    \le \||H|^2\|_{L^2}^2 \cdot \|\mathds{1}_{B_1(0)}\|_{L^1}^2 = \mathfrak{d}(0,|H|)^4 \cdot |B_1(0)|^2.
    $$
    It remains to show that $C_P(w\chi,w)<\infty$ if $F\neq 0$ and $w=(|F|^2\ast\gamma)^2$, because then the statement follows from Theorem \ref{thm:stabilitystft}.
    To that end, let $\Omega\subseteq \R^2$ be 
     open, connected and bounded with smooth boundary such that $\Omega\supseteq K$. Then, as $w\chi\le w\mathds{1}_{\Omega} \le w$ we have that 
    $$
    C_P(w\chi,w) \le C_P(w \mathds{1}_{\Omega}) \le \frac{\sup_{x\in \Omega}{w(x)}}{\inf_{x\in \Omega}w(x)} C_P(\mathds{1}_{\Omega}).
    $$
    Note that $w=(|F|^2\ast\gamma)^2$ is continuous, bounded and strictly positive. Thus, $\inf_{x\in \Omega}w(x)>0$ and $\sup_{x\in \Omega}w(x)<\infty$. 
    Since bounded, connected domains with sufficiently regular boundary support unweighted Poincar\'e inequalities   \cite{evans:pdes}, we have that $C_P(\mathds{1}_\Omega)<\infty$ and are done.
\end{proof}

Stability for global phase reconstruction is a more delicate affair and requires an argument which is custom-tailored to the concrete function under consideration. Exemplarily
we establish that if function and window are the same ($f=g$)  with $g$ either window from Lemma \ref{lem:expadmpairs}, then the resulting global Poincar\'e constant is in fact finite.
\begin{theorem}\label{thm:finitepoinc1}
    Let $g(t)=\exp(t-e^t)$, let $\gamma:\R^2\to (0,\infty)$ be an integrable and log-concave weight (e.g. $\gamma(x,\xi)=e^{-a|x|-b|\xi|}$, $a,b>0$).
    Then $w = (|\V_g g|^2\ast\gamma)^2$ is log-concave and $C_P(w)<\infty$.
\end{theorem}
In the case of the one-sided exponential, we need to employ a non-trivial weight $\chi$ to compensate for the slow decay of the weight w.r.t. the frequency variable. 
\begin{theorem}\label{thm:finitepoinc2}
Let $\chi(x,\xi)=\frac1{1+\xi^2}$, let $\gamma(x,\xi)=\exp(-a|x|-b|\xi|)$ with $a,b>0$, let $g(t)=e^{-t}\mathds{1}_{(0,\infty)}(t)$ and let $w=(|\V_g g|^2\ast\gamma)^2$.
Then $C_P(w\chi,w)<\infty$.
\end{theorem}
The proofs of Theorem \ref{thm:finitepoinc1} and Theorem \ref{thm:finitepoinc2} are postponed to Section \ref{sec:finitepoinc}.

\subsection{Discussion and open question}

\subsubsection{Extension to the multivariate case}
We have restricted ourselves to the univariate setting in Theorem \ref{thm:stabilitystft}.
However, it is straight forward to extend the result to the multivariate setting for the case where the window $g$ is chosen to be a tensor product 
$g=g_1\otimes\ldots\otimes g_d$ with $(g_1,\gamma_1),\ldots,(g_d,\gamma_d)$ admissible pairs.
This is due to the fact that in this case the ambiguity function is precisely the tensor product of the respective one-dimensional ambiguity functions, i.e., 
$$
\V_g g = \V_{g_1}g_1 \otimes \ldots \otimes \V_{g_d}g_d.
$$

\subsubsection{SLPR beyond STFT}So far we only know how to establish $\Gamma$-SLPR when $V=\V_g (L^2(\R))$ arises as the image of the STFT (and for very particular windows $g$ only). For that case we have rather explicit knowledge what the relevant extension operator looks like (see Section \ref{sec:slprforstft}).
From this perspective it is natural to ask:
\begin{quote}
Can the general result, Theorem \ref{thm:abstractstability} be applied beyond the STFT setting?
Can the property of SLPR be established for any other natural function classes $V$?
\end{quote}

\subsubsection{Concentration implies finite Poincar\'e constant?}
Any weight $w$ which supports a Poincar\'e inequality has exponential concentration, cf. \cite[Theorem 2]{lecturenotes:concofmeas}.
The converse is of course not true: Just consider a weight which has compact but disconnected support. In our situation such weights are however not permitted as $w=(|F|^2\ast\gamma)^2$ is strictly positive.
One may therefore ask:
\begin{quote}
    Is exponential concentration of $F=\V_g f$ sufficient for $w=(|F|^2\ast\gamma)^2$ to support a Poincar\'e inequality?
\end{quote}
More formally, this can be phrased as follows:
\begin{conj}
    Let $g(t)=\exp(t-e^t)$, and let $\gamma(x,\xi)=\exp(-a|x|-b|\xi|)$, $a,b>0$.
    Suppose that $f\in L^2(\R)$ has exponential time-frequency concentration, i.e.,
    $$
    \exists \delta>0:\quad |\V_g f(z)|\cdot e^{\delta|z|} \in L^\infty(\R^2)
    $$
    and let $w=(|\V_g f|^2\ast\gamma)^2$.
    Then $C_P(w)<\infty$.
\end{conj}

Note that for $g$ the one-sided exponential the above conjecture is an empty statement since there are no functions $f$ (except for the zero function) for which the spectrogram $|\V_g f|$ has exponential concentration as the window is discontinuous.
This is why the conjecture has been formulated explicitly for $g(t)=\exp(t-e^t)$.

\subsection{Notation}
We say that $w$ is a weight on a domain $\Omega$ if $w$ is a measurable and non-negative function on $\Omega$.
Given a function $f$ on $\R^d$ we denote its reflection by $\mathcal{R}f=f(-\cdot)$.
The indicator function of a set $A$ is denoted by $\mathds{1}_A$.
By $B_r(x)\subseteq\R^d$ we denote the open ball of radius $r$ centered at $x\in\R^d$.
We use the notation $f\otimes g$ for the tensor product, that is, $(f\otimes g)(x,y)=f(x)g(y)$.
For $f,g:\Omega\to \R_+$ we write $f\lesssim g$ if there exists a finite $c>0$ such that $f(x)\le c g(x)$ for all $x\in\Omega$.
We use the notation $f\asymp g$ in case $f\lesssim g$ and $g\lesssim f$.\\
The Fourier transform of $f\in L^1(\R^d)\cap L^2(\R^d)$ is defined by 
$$
\hat{f}(\xi)=(\mathcal{F}f)(\xi) = \int_{\R^d} f(x)e^{-2\pi i \xi x}\,\mbox{d}x,\quad \xi\in\R^d,
$$
and is extended to $L^2(\R^d)$ in the usual way.
Given $p\in[1,\infty)$ and $\chi:\R^d\to \R_+$ measurable  we denote the weighted $L^p$-norm of a measurable function $F$ by
$$
\|F\|_{L^p(\chi)}= \left( \int_{\R^d} |F(x)|^p \chi(x)\,\mbox{d}x \right)^{1/p}
$$
Occasionally, there can appear either a set $\Omega\subseteq \R^d$ or a measure $\mu$ instead of $\chi$ in the argument.
Perhaps unsurprisingly, this is then to be understood in the following ways:
$$
\|F\|_{L^p(\Omega)}=\|F\cdot \mathds{1}_\Omega\|_{L^p},\quad \text{and} \quad \|F\|_{L^p(\mu)} = \left( \int |F(x)|^p\,\mbox{d}\mu(x)\right)^{1/p}.
$$
\section{Preliminaries}\label{sec:prelims}

\subsection{Key lemma}
 
The following innocent result will play a rather crucial role.
It provides us with a quantitative link between standard phase retrieval and its lifted reformulation.
\begin{lemma}\label{lem:hilbert2}
    For all $\phi,\psi\in L^2(\R^d)$ it holds that 
     $$
    \frac{\|\phi\|_{L^2}^2 + \|\psi\|_{L^2}^2}2 \cdot \min_{|\lambda|=1} \|\psi-\lambda\phi\|_{L^2}^2 \le \|\psi\otimes\overline{\psi}-\phi\otimes\overline{\phi}\|_{L^2}^2.
    $$
\end{lemma}
\begin{proof}
    We define three non-negative numbers by 
    $a=\|\phi\|_{L^2}^2$, 
    $b=\|\psi\|^2_{L^2}$  and 
    $c=|\langle \psi,\phi\rangle_{L^2}|$.
    Elementary computations show that 
    \begin{equation*}
        \min_{|\lambda|=1} \|\psi-\lambda\phi\|_{L^2}^2 = a+b-2c
    \end{equation*}
    and that
    \begin{equation*}
        \|\psi\otimes\overline{\psi}-\phi\otimes\overline{\phi}\|_{L^2}^2 = a^2+b^2-2c^2.
    \end{equation*}
    As $(a+b)^2\le 2(a^2+b^2)$ and $a+b\ge 2c$ we get that 
    $$
        \frac{a+b}2 \cdot \left(a+b-2c\right)
        \le a^2 + b^2-2c^2-c(a+b-2c)
        \le a^2+b^2-2c^2
    $$
    and are done.
\end{proof}

\subsection{Prerequisites on Poincar\'e inequalities}

\subsubsection{Estimates for Poincar\'e constants}\label{subsec:estspoincconst}
In this section we collect a couple of sufficient conditions on the weights which guarantee that a Poincar\'e inequality holds.\\
First up are log-concave weights. Recall that a function $f:\Omega\to \R_+$ with $\Omega\subseteq\R^d$ a convex domain is called log-concave if 
$\log f$ is concave. Perhaps the simplest example of a log-concave function is a Gaussian.
The following classical result states that log-concave weights support Poincar\'e inequalities.
\begin{theorem}[Bobkov \cite{bobkov:isoplogconcave}]\label{thm:bobkovlogconc}
Let $w:\R^d\to \R_+$ be log-concave weight which induces a probability measure (i.e., $\int_{\R^d} w(x)\,\mbox{d}x=1$).
There exists a universal constant $K>0$ (independent of $w$ and $d$) such that 
$$
C_P(w) \le K \|x\mapsto |x-x_0|\|_{L^2(w)}^2,
$$
where $x_0=\int_{\R^d} x\,w(x)\mbox{d}x$ is the barycenter of the measure.
\end{theorem}

Next we consider  an example of a weight which is not log-concave: The Cauchy distribution has the density 
\begin{equation}\label{def:cauchydensity}
    w_\beta(x) =\frac1{Z} \left(1+|x|^2\right)^{-\beta}, \quad x\in\R^d,\, \beta>\frac{d}2,
\end{equation}
where $Z$ is the appropriate normalization factor that turns $w_\beta$ into a probability distribution.

\begin{theorem}[Bobkov and Ledoux \cite{bobkov:cauchy}]\label{bobkov:cauchy}
    If $\beta\ge d+1$, then $C_P\left( \frac{w_\beta(x)}{1+|x|^2}, w_\beta(x)\right) \le \frac1{2\beta}$. 
\end{theorem}

The next result tells us how Poincar\'e constants behave under tensorisation.
\begin{lemma}\label{lem:tensorpoincare}
    Let $v_1,v_2$ be densities of probability measures on $\R^{d_1}$ and $\R^{d_2}$, respectively.
    Let $w_1,w_2$ be weights such that $v_1\le w_1$ and $v_2\le w_2$.
    Then it holds that 
    $$C_P(v_1\otimes v_2, w_1\otimes w_2) \le \max\{C_P(v_1,w_1), C_P(v_2,w_2)\}.$$
\end{lemma}
For the symmetric case (i.e., $v_1=w_1$ and $v_2=w_2$) the above  inequality is well-known, cf. \cite[Theorem 4.1]{bonnefont:lecturenotes}.
The general asymmetric case follows by some minor adaptations.
For the sake of completeness we give a full proof.
\begin{proof}[Proof of Lemma \ref{lem:tensorpoincare}]
For $k\in\{1,2\}$, let us denote $\mbox{d}\mu_k(x)=w_k(x)\mbox{d}x$ and $\mbox{d}\nu_k(x)=v_k(x)\mbox{d}x$.
Moreover, set $\mu=\mu_1\otimes\mu_2$ and $\nu=\nu_1\otimes\nu_2$.\\
Note that if $\sigma$ is a probability measure on some measure space and $u\in L^2(\sigma)$, then 
$$
\inf_{c} \|u-c\|_{L^2(\sigma)}^2 
=
\left\|u-\langle u,1\rangle_{L^2(\sigma)} 1\right\|_{L^2(\sigma)}^2
=
\|u\|_{L^2(\sigma)}^2- \left( \int u\,\mbox{d}\sigma \right)^2.
$$ 
Let $u\in C_b^\infty(\R^{d_1}\times\R^{d_2},\R)$ arbitrary.
With this, rewrite  
\begin{align*}
    (\ast) := \inf_{c} \|u- &c\|_{L^2(\nu)}^2 \\
    &=
    \int \left(\int u(x,y)^2\,\mbox{d}\nu_1(x)\right)\,\mbox{d}\nu_2(y) - \left( \iint u(x,y)\,\mbox{d}\nu_1(x)\,\mbox{d}\nu_2(y)\right)^2\\
    &=  \int \left(\int u(x,y)^2\,\mbox{d}\nu_1(x)\right)\,\mbox{d}\nu_2(y) -  \int \left(\int u(x,y)\,\mbox{d}\nu_1(x)\right)^2\,\mbox{d}\nu_2(y)\\
    &\quad +  \int \left(\int u(x,y)\,\mbox{d}\nu_1(x)\right)^2\,\mbox{d}\nu_2(y)- \left( \iint u(x,y)\,\mbox{d}\nu_1(x)\,\mbox{d}\nu_2(y)\right)^2.
\end{align*}
With the notation $u_y(x)=u(x,y)$ and $\phi(y)=\int u(\cdot,y)\,\mbox{d}\nu_1$, we further obtain
\begin{align*}
    (\ast) &= \int \left( \|u_y\|_{L^2(\nu_1)}^2-\left( \int u_y \mbox{d}\nu_1  \right)^2 \right) \mbox{d}\nu_2(y)
    + \left(\|\phi\|_{L^2(\nu_2)}^2 - \left(\int \phi\mbox{d}\nu_2 \right)^2\right)\\
    &= \int\left( \inf_{c} \|u_y - c\|_{L^2(\nu_1)}^2
    \right) \mbox{d}\nu_2(y) + \inf_c \|\phi-c\|_{L^2(\nu_2)}^2.
\end{align*}
Next we apply the respective Poincar\'e inequality to each of the terms and get that 
\begin{equation}\label{eq:estafterpoinc}
    \inf_c\|u-c\|_{L^2(\nu)}^2 \le C_P(v_1,w_1) \int \|\nabla u_y\|_{L^2(\mu_1)}^2\,\mbox{d}\nu_2(y) + C_P(v_2,w_2) \|\nabla\phi\|_{L^2(\mu_2)}^2.
\end{equation}
Now it suffices to prove that 
\begin{equation}\label{eq:relationgradients}
    \int \|\nabla u_y\|_{L^2(\mu_1)}^2\,\mbox{d}\nu_2(y) + \|\nabla\phi\|_{L^2(\mu_2)}^2 \le \|\nabla u\|_{L^2(\mu)}^2.
\end{equation}
Indeed, \eqref{eq:relationgradients} implies that \eqref{eq:estafterpoinc} is bounded by $\max\{C_P(v_1,w_1), C_P(v_2,w_2)\} \|\nabla u\|_{L^2(\mu)}$, which implies the statement of the lemma. \\

We use the notation $\nabla_x u=\left(\frac{\partial u}{\partial x_1},\ldots,\frac{\partial u}{\partial x_{d_1}}  \right)^T$, and $\nabla_y u$ is defined accordingly. 
Since $v_2\le w_2$ we have that 
\begin{equation}\label{eq:estfirstgradient}
\int \|\nabla u_y\|_{L^2(\mu_1)}^2\,\mbox{d}\nu_2(y) \le \|\nabla_x u\|_{L^2(\mu_1)}^2.
\end{equation}
For the second term, interchanging order of integration and differentiation shows that 
$\nabla\phi(y) = \int \nabla_y u(x,y)\,\mbox{d}\nu_1(x)$.
By Jensen's inequality,
$$
|\nabla\phi(y)|^2 \le \int |\nabla_y u(x,y)|^2 \mbox{d}\nu_1(x).
$$
Integrating both sides w.r.t. $\mu_2$ and using that $v_1\le w_1$, implies that 
$$
\|\nabla\phi\|_{L^2(\mu_1)}^2 \le \|\nabla_y u\|_{L^2(\mu)}^2.
$$
This, together with \eqref{eq:estfirstgradient} gives \eqref{eq:relationgradients}, and we are done.
\end{proof}

\begin{remark}[Normalisation]
    Later on we will use the results of this section to show that certain weights support Poincar\'e inequalities.
    This means, we are interested in configurations which yield finite Poincar\'e constants without caring too much about the concrete values.
    For this purpose the assumption that $w(x)\mbox{d}x$ is a probability measure can always be neglected by rescaling.
\end{remark}

\subsubsection{A Poincar\'e inequality without derivatives}

The objective of this paragraph is to derive a version of a Poincar\'e inequality which can be applied to functions with little regularity.
As a substitute for the length of the gradient we introduce the following concept. 
\begin{definition}
    The \emph{local deviation} of $u\in L_{loc}^2(\R^d)$ is defined by
    \begin{equation}
        \delta_1[u](x)\coloneqq \left( \int_{B_1(x)} |u(y)-u(x)|^2\,\mbox{d}y\right)^{1/2},\quad x\in\R^d.
    \end{equation}
\end{definition}
What we are interested in is the following version of a Poincar\'e inequality.
\begin{definition}
    Let $v,w$ be weights on $\R^d$ and assume that $v\in L^1(\R^d)$. The \emph{modified Poincar\'e} constant, denoted by $C_P^*(v,w)$ is the smallest possible constant $C>0$ such that 
    \begin{equation}
        \forall u\in L^\infty(\R^d):\quad
        \inf_{c\in\C} \|u-c\|_{L^2(v)}^2 \le C \| \delta_1[u]\|_{L^2(w)}^2. 
    \end{equation}
\end{definition}

As we shall see next, if $v,w$ support a Poincar\'e inequality (in the ordinary sense) then $C_P^*(v,w)$ is finite as well. 
\begin{theorem}\label{thm:poincmod}
Let $v,w$ be weights on $\R^d$, and assume that 
$v\in L^1(\R^d)$ and that $v\le w$ pointwise.
There exists a finite constant $c$ (depending on $d$ only) such that 
\begin{equation}
    C_P^*(v,w) \le c \cdot (1+C_P(v,w)).
\end{equation}
\end{theorem}

\begin{remark}[Motivation for looking at modified Poincar\'e inequalities]
The ordinary Poincar\'e inequality is ultimately about the question 
    'To what extent is the norm of the derivative indicative about how close a function is to a constant?'. 
    Figure \ref{fig:fctderiv} shows 
    two basic function on an interval of the real line. 
    The more interesting one is the oscillatory function: Its derivative will have a relatively large norm due to the oscillations while the function is close to a constant. This suggests that the quantity $\|\nabla u\|_{L^2}$ on the right hand side of a Poincar\'e inequality is unnecessarily excessive in this case.
    Theorem \ref{thm:poincmod} can be understood as an approval of this intuition as it effectively allows us to replace $\|\nabla u\|_{L^2}$ by $\|\delta_1[u]\|_{L^2}$.
\begin{figure}
    \centering
    \includegraphics[width=0.4\textwidth]{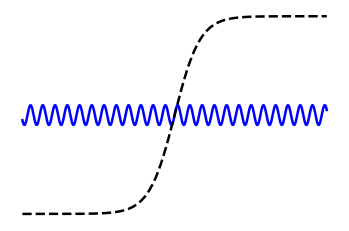}
    \caption{ Example of a purely oscillatory function with a small amplitude (blue, solid line) and a function which actually deviates much from its mean (black, dashed line).
    }\label{fig:fctderiv}
\end{figure}
\end{remark}

\begin{proof}[Proof of Theorem \ref{thm:poincmod}]
    Let $u\in L^\infty(\R^d)$ be an arbitrary function.
    Moreover, let $\phi\in C^\infty(\R^d)$ be a bump function, which satisfies that $\suppp \phi\subseteq B_1(0)$ and that $\int_{\R^d}\phi(x)\,\mbox{d}x=1$.
    We set $\tilde{u}\coloneqq u\ast \phi$ and notice that $\tilde{u}\in C_b^\infty(\R^d)$.
    We apply the (ordinary) Poincar\'e inequality on $\tilde{u}$ and get that 
    \begin{align*}
        \inf_{c\in\C} \|u-c\|_{L^2(v)}^2 &\le 2\left(\|u-\tilde{u}\|_{L^2(v)}^2 +   \inf_{c\in\R} \|\tilde{u}-c\|_{L^2(v)}^2\right)\\
        &\le 2 \left(\|u-\tilde{u}\|_{L^2(w)}^2 + C_P(w)\|\nabla \tilde{u}\|_{L^2(w)}^2 \right)
    \end{align*}
    We will slightly abuse notation and use the letter $w$ for the measure $w(x)\mbox{d}x$ induced by the function $w$.
    An application of Cauchy-Schwarz allows us to bound the first term on the right as follows.
    \begin{align*}
        \|u-\tilde{u}\|_{L^2(w)}^2 
        &= \int_{\R^d} |u(x)-\tilde{u}(x)|^2\,\mbox{d}w(x)\\
        &= \int_{\R^d} \left| \int_{B_1(x)} (u(x)-u(y)) \phi(x-y)\,\mbox{d}y \right|^2 \,\mbox{d}w(x)\\
        &\le \int_{\R^d} \delta_1[u](x)^2 \cdot \|\phi\|_{L^2}^2\,\mbox{d}w(x)\\
        &= \|\delta_1[u]\|_{L^2(w)}^2\cdot \|\phi\|_{L^2}^2 
    \end{align*}
    To deal with the second term, observe that 
    $$
    \forall \ell\in\{1,\ldots,d\}:\quad \int_{\R^d} \frac{\partial \phi}{\partial x_\ell}\,\mbox{d}x = 0.
    $$
    Denoting $\phi_\ell = \frac{\partial \phi}{\partial x_\ell}$, we therefore have that $k\ast\phi_\ell=0$ whenever $k$ is a constant function. Thus, 
    \begin{align*}
        \|\nabla \tilde{u}\|_{L^2(w)}^2 
        &= \sum_{\ell=1}^d \int_{\R^d} \left| \frac{\partial \tilde{u}}{\partial x_\ell} (x)\right|^2\,\mbox{d}w(x)\\
        &= \sum_{\ell=1}^d \int_{\R^d} |(u\ast \phi_\ell)(x)|^2\,\mbox{d}w(x)\\
        &= \sum_{\ell=1}^d \int_{\R^d} \left|\int_{\R^d} (u(y)-u(x))\phi_\ell(x-y)\,\mbox{d}y\right|^2\,\mbox{d}w(x)\\
        &\le \sum_{\ell=1}^d \int_{\R^d} \delta_1[u](x)^2 \cdot \|\phi_\ell\|_{L^2}^2\,\mbox{d}w(x)\\
        &= \|\delta_1[u]\|_{L^2(w)}^2 \cdot \|\nabla \phi\|_{L^2}^2,
    \end{align*}
    where we used Cauchy-Schwarz once more.
    Combining the pieces, gives that 
    \begin{align*}
        \inf_{c\in\R}\|u-c\|_{L^2(w)}^2 &\le 2 \left(\|\phi\|_{L^2}^2 + C_P(w) \|\nabla\phi\|_{L^2}^2 \right) \|\delta_1[u]\|_{L^2(w)}^2\\
        &\le 2 \|\phi\|_{H^1}^2 (1+C_P(w)) \|\delta_1[u]\|_{L^2(w)}^2,
    \end{align*}
    which implies the claim.
\end{proof}

\subsection{SLPR for the STFT}\label{sec:slprforstft}
The upcoming statement reveals how differences in the lifted setting can be expressed in terms of the respective ambiguity functions.
\begin{lemma}\label{lem:planchshift}
    Suppose that $f,g,h\in L^2(\R)$ and let us denote $F=\V_g f$ and $H=\V_g h$.
    For all $\tau\in\R^{2}$ it holds that 
    \begin{equation}\label{eq:planchshift}
        \|F(\cdot)\overline{F(\cdot+\tau)}-H(\cdot)\overline{H(\cdot+\tau)}\|_{L^2} = \|[\V_f f - \V_h h](\cdot) \V_g g(\cdot+\tau)\|_{L^2}. 
    \end{equation}
\end{lemma}
\begin{proof}
    As a starting point we use the following fact which may be found in \cite{groechenig:mystery}. Let $f,g,\phi,\psi\in L^2(\R)$ and let $\mathcal{J}=\begin{psmallmatrix} 0& 1\\ -1 & 0\end{psmallmatrix}$. Then 
    \begin{equation}\label{eq:formftpostft}
    \left( \V_g f \cdot \overline{\V_\psi \phi} \right)^\wedge (\zeta) = \V_\phi f(-\mathcal{J}\zeta) \cdot \overline{\V_\psi g(-\mathcal{J}\zeta)}, \quad \zeta\in\R^2.
    \end{equation}
    Fix $\tau=(p,q)\in\R^2$ and choose 
$\phi = M_{-q}f$ and $\psi = T_p g$.
Notice that
\begin{equation}
    \V_\psi \phi(z) = \langle M_qf, M_{z_2} T_{z_1} T_p g\rangle
    = \langle f, \pi(z+\tau)g\rangle = \V_g f(z+\tau),
\end{equation}
and that (with $\zeta=(u,v)$)
\begin{equation}
    \V_\phi f (-\mathcal{J}\zeta)=\langle f, M_u T_{-v} M_{-q}f\rangle 
    = \langle f, e^{-2\pi i q\cdot v}M_{u-q}T_{-v}f\rangle 
    =
    e^{2\pi i q\cdot v} \V_f f(-v,u-q),
\end{equation}
and further, that 
\begin{equation}
    \V_\psi g(-\mathcal{J}\zeta) = \langle g, M_u T_{-v}T_p g\rangle = \V_g g(p-v,u).
\end{equation}
Thus, equation \eqref{eq:formftpostft} implies that
\begin{equation*}
    \left( \V_g f(\cdot) \overline{\V_g f(\cdot+\tau)} \right)^\wedge(\zeta)
    = e^{2\pi iq\cdot v} \cdot \V_f f(-v,u-q) \cdot \overline{\V_g g(-v+p,u)}
\end{equation*}
Clearly $f$ may be replaced by $h$. The statement now follows by taking differences, by applying Plancherel's formula and by a change of variable.
\end{proof}
For window functions $g$ 
such that $|\V_g g|$ is a TSW we can deduce the following statement.

\begin{proposition}\label{prop:rhostableprstft}
    Suppose that $g\in L^2(\R)$ is such that $|\V_g g|$ is a TSW and let $\mu:\R^2\to \R_+$ be locally bounded such that 
    $$
    |\V_g g(z+\tau)|\le \mu(\tau) |\V_g g(z)|,\quad z,\tau\in\R^2.
    $$
    Moreover, assume that $\sigma$ is a weight on $\R^2$ such that  
    $\mu\in L^2(\sigma)$.\\
    Then $V=\V_g (L^2(\R))$ does $\rho$-SLPR with $\rho(z,z')=\sigma(z'-z)$.
    In particular, if $(g,\gamma)$ is an admissible pair, then $V$ does $\Gamma$-SLPR with $\Gamma=\Gamma[\gamma]$.
\end{proposition}

We require the following auxiliary result.
\begin{lemma}\label{lem:estsclicesdiag}
    Assume the conditions of Proposition \ref{prop:rhostableprstft}, let $f,h\in L^2(\R)$ and denote $F=\V_g f$ and $H=\V_g h$. For all $\tau\in\R^d$ it holds that
    \begin{equation}
        \|F(\cdot)\overline{F(\cdot+\tau)}-H(\cdot)\overline{H(\cdot+\tau)}\|_{L^2} \le \mu(\tau) \cdot \||F|^2-|H|^2\|_{L^2}.
    \end{equation}
\end{lemma}
\begin{proof}
    According to Lemma \ref{lem:planchshift}, the left hand side squared is 
    \begin{align*}
        \|[\V_f f - \V_h h](\cdot)\V_g g(\cdot+\tau)\|_{L^2}^2
        &= \int\limits_{\R^d} |[\V_f f-\V_h h](z)|^2 \cdot |\V_g g(z+\tau)|^2\,\mbox{d}z\\
        &\le \int\limits_{\R^d} |[\V_f f-\V_h h](z)|^2 \cdot \mu(\tau)^2 |\V_g g(z)|^2\,\mbox{d}z\\
        &= \mu(\tau)^2 \cdot  \|[\V_f f-\V_h h] \V_g g\|_{L^2}^2.
    \end{align*}
    Applying Lemma \ref{lem:planchshift} once more and taking square roots implies the claim.
\end{proof}
Now, the only thing that is left is to integrate things up in the correct way.
\begin{proof}[Proof of Proposition \ref{prop:rhostableprstft}]
    Let $F=\V_g f$ and $H=\V_g h$ for $f,h\in L^2(\R)$.
    By Fubini, as $\rho(z,z+\tau)=\sigma(\tau)$ and with Lemma \ref{lem:estsclicesdiag} we get
    \begin{align*}
        \|F\otimes\overline{F}-H\otimes\overline{H}\|_{L^2(\rho)}^2 
        &= \int\limits_{\R^2} \int\limits_{\R^2}\left| F(z)\overline{F(z+\tau)}-H(z)\overline{H(z+\tau)} \right|^2 \,\mbox{d}z \,\sigma(\tau)\,\mbox{d}\tau\\
        &\le \||F|^2-|H|^2\|_{L^2}^2 \cdot \int\limits_{\R^2} \mu(\tau)^2\sigma(\tau)\mbox{d}\tau.
    \end{align*}
    This implies the assertion.
\end{proof}

\subsection{Identifying admissible pairs}\label{sec:prooflemmaadmpairs}
The goal of this section is to prove Lemma \ref{lem:expadmpairs}.
Recall that for $(g,\gamma)$ to be an admissible pair, the modulus of the ambiguity function $|\V_g g|$ necessarily has to be a TSW. In particular, $\V_g g$ cannot have any zeros.
The question of which windows yield zero-free ambiguity functions has been investigated by Gröchenig, Jaming and Malinnikova \cite{grochenig:zeros}. 
The examples we consider here are based on their findings.
\begin{proof}[Proof of Lemma \ref{lem:expadmpairs}]
    For each case we will follow the same pattern: 
    First, identify a function $\mu$ such that $|\V_g g(z+\tau)|\le \mu(\tau)\mu(z)$ in order to derive that $|\V_g g|$ is a TSW.
    Secondly, pick $\gamma$ accordingly in order to make sure that 
    $\mu\in L^2(\gamma\ast\mathcal{R}\gamma)$.\\

    \textbf{Case (i):} $g(t)=\exp(t-e^t)$ and $\gamma(x,\xi)=e^{-a|x|-b|\xi|}$, where $a>2\pi^2$ and $b>2$.
    The ambiguity function has been explicitly computed in \cite[Example 4]{grochenig:zeros}, and its modulus is given by
    \begin{equation}\label{eq:modVgg1}
        |\V_g g(x,\xi)| = \frac14  \sech\left(\frac{x}2\right)^2 |\Gamma(2-2\pi i \xi)|.
    \end{equation}
    We will make use of the fact that $\sech \asymp e^{-|\cdot|}$ to estimate 
    \begin{equation}
        \sech\left(\frac{x+t}2\right)^2 \lesssim e^{- |x+t| } \le e^{|t|-|x|}\lesssim e^{|t|} \sech\left(\frac{x}2\right)^2.
    \end{equation}
    Moreover, with $\frac{t}{\sinh(t)} \asymp (1+|t|)e^{-|t|}$ and the identity $|\Gamma(2+bi)|^2= \frac{\pi b}{\sinh(\pi b)} (1+b^2)$ we get 
    \begin{align*}
    |\Gamma(2-2\pi i (y+t))|^2 &= \frac{2\pi^2 (y+t)}{\sinh(2\pi^2(y+t))}(1+4\pi^2 (y+t)^2) \\
    &\lesssim (1+2\pi^2|y+t|) e^{-2\pi^2|y+t|}
    (1+4\pi^2(y+t)^2)\\
    &\lesssim (1+|y+t|^3) e^{-2\pi^2 |y+t|}\\
    &\lesssim (1+|y|^3)(1+|t|^3) e^{2\pi^2|t|} e^{-2\pi^2|y|}\\
    &\le (1+|t|^3)e^{2\pi^2|t|} \cdot \frac{2\pi^2 y}{\sinh(2\pi^2 y)} (1+4\pi^2 y^2)\\
    &= (1+|t|^3)e^{2\pi^2|t|} \cdot |\Gamma(2-2\pi i y)|^2.
\end{align*}
Combining the two estimates implies that there exists a numerical constant $c>0$ such that
\begin{equation}
\mu(\tau) = c (1+|\tau_2|^\frac32) e^{\pi^2 |\tau_2|+|\tau_1|},\quad \tau=(\tau_1,\tau_2)\in\R^2
\end{equation}
satisfies $|\V_g g(z+\tau)|\le \mu(\tau)|\V_g g(z)|$.
Hence, $|\V_g g|$ is a TSW and it remains to check that $\mu$ satisfies the integrability condition.
To that end, compute
    \begin{equation}\label{eq:gammaRgammaconv}
        (\gamma\ast\mathcal{R}\gamma)(x,\xi)= 
        (a^{-1}+|x|)e^{-a|x|} \cdot (b^{-1}+|\xi|)e^{-b|\xi|}.
    \end{equation}
    With this we see that 
    \begin{equation*}
        \mu \in L^2(\gamma\ast\mathcal{R}\gamma) \quad \Leftrightarrow \quad \mu^2 \cdot (\gamma\ast\mathcal{R}\gamma) \in L^1 \quad  
         \Leftrightarrow \quad a>2\pi^2 \quad\text{and}\quad b>2,
    \end{equation*}
    which settles case (i).\\

    \textbf{Case (ii):} $g(t)= e^{-t}\mathds{1}_{(0,\infty)}(t)$ and $\gamma(x,\xi)=e^{-a|x|-b|\xi|}$.\\
    The modulus of the ambiguity function of the one-sided exponential is given by 
    $$
    |\V_g g(x,\xi)| = \frac{e^{-|x|}}{2|1+i\pi\xi|},
    $$
    see \cite[Example 3]{grochenig:zeros}.
    By triangle inequality, for all $\xi,\tau_2\in\R$,
    $$
    |1+i \pi \xi| \le |1+i\pi(\xi+\tau_2)| + \pi|\tau_2| \le |1+i\pi(\xi+\tau_2)| \cdot (1+\pi|\tau_2|).
    $$
    Hence, with $z=(x,\xi), \tau=(\tau_1,\tau_2)$,
    $$
    |\V_g g(z+\tau)| \le (1+\pi|\tau_2|)e^{-|\tau_1|} |\V_g g(z)|.
    $$
    It remains to verify that $\mu(\tau)=e^{-|\tau_1|}(1+\pi|\tau_2|)$ is square integrable w.r.t. $(\gamma\ast\mathcal{R}\gamma)$.
    Using formula \eqref{eq:gammaRgammaconv} once more, we see that this is the case if and only if $a>2$ and $b>0$.
    This settles case (ii).
\end{proof}

\begin{remark}[Remaining candidates]
Besides $\exp(t-e^t)$ and the one-sided exponential, the article \cite{grochenig:zeros} contains further examples of functions whose ambiguity function is nowhere vanishing. 
One class of examples are Gaussians.
The other examples are all built from one-sided exponentials. That is, with 
$\eta_a(t) := e^{-at} \mathds{1}_{[0,\infty)}(t)$,
all of the following choices
\begin{itemize}
    \item $g(t)=e^{-ct^2}$ with $c>0$
    \item $g(t)=\eta_a\ast\eta_b$ with $a,b>0$, $a\neq b$
    \item $g(t)=\eta_a\ast \mathcal{R}\eta_b$ with $a,b>0$, $a\neq b$
    \item $g(t)=t^n \eta_a(t)$ with $n\in\mathbb{N}$
\end{itemize}
have the property that $\V_g g$ has no zeros. If $g$ is a Gaussian, then $|\V_g g|$ is a Gaussian and hence cannot be a TSW. 
The respective ambiguity functions of the other candidates are also explicitly available but unfortunately take a rather complicated form.
For these cases we did not manage to prove (and neither disprove) that $|\V_g g|$ is a TSW.
\end{remark}
 
\section{Proofs}\label{sec:proofs}

\subsection{The abstract stability result}\label{ref:proofabstractstability}
The first goal of this section is to establish the following result.
\begin{proposition}\label{prop:abstractstability}
    Let $\chi:\R^d\to [0,1]$ be a weight and let 
    $\gamma\in L^1(\R^d)$ be a TSW.\\
    There exists $c>0$ (depending on $\gamma, \chi$ and $d$ only) such that for all $F,H\in L^4(\R^d)$
    \begin{equation}
        \inf_{c\in\C} \|H-cF\|_{\mathcal{L}(\chi)}^4
        \le 
        c (1+C_P(w\chi,w)) \|H\otimes \overline{H}-F\otimes\overline{F}\|_{L^2(\Gamma)}^2,
    \end{equation}
    with $w=(|F|^2\ast \gamma)^2$ and $\Gamma=\Gamma[\gamma]$.
\end{proposition}

\begin{proof}
Throughout $C$ will denote a constant which depends on $\gamma$ and $d$ only but may change from one line to another.\\
Let us begin with some notation: Given $\tau\in\R^d$ we introduce weighted versions of $F$ and $H$ by 
$$
F_\tau = F(\cdot)\sqrt{\gamma(\tau-\cdot)},\quad\text{and}\quad
H_\tau = H(\cdot)\sqrt{\gamma(\tau-\cdot)}.
$$
Note that $w(\tau)=\|F_\tau\|_{L^2}^4$.
Moreover, let $u:\R^d\to \mathbb{T}$ be such that 
\begin{equation}
    \forall \tau\in\R^d:\quad u(\tau)\in \argmin_{c\in\mathbb{T}} \|H_\tau-c F_\tau\|_{L^2}.
\end{equation}
For all $G\in L^4(\R^d)$ we have that 
\begin{align*}
    \|G\|_{\mathcal{L}(\chi)} &= \| \tau\mapsto \|G\|_{L^2(B_1(\tau))} \|_{L^4(\chi)}\\
    &\le \frac1{\inf_{y\in B_1(0)} \gamma(y)} 
    \| \tau\mapsto \|G\|_{L^2(\gamma(\tau-\cdot))} \|_{L^4(\chi)}.
\end{align*}
The positivity and continuity of $\gamma$ ensure that the $\inf$ is strictly positive.
Hence,
\begin{multline*}
    \inf_{c\in\C} \|H-cF\|_{\mathcal{L}(\chi)} 
        \le C \cdot \inf_{c\in\C} \| \tau\mapsto \|H_\tau-c F_\tau\|_{L^2} \|_{L^4(\chi)} \\
        \le C \left( \|\tau\mapsto \|H_\tau-u(\tau)F_\tau\|_{L^2}\|_{L^4(\chi)} + \inf_{c\in\C} \|\tau \mapsto \|(u(\tau)-c)F_\tau\|_{L^2} \|_{L^4(\chi)} \right).
\end{multline*}
Since $u$ attains its values on the unit circle exclusively, we get the following bound on the second term on the right.
\begin{align*}
    \inf_{c\in\C} \|\tau \mapsto \|(u(\tau)-c)F_\tau\|_{L^2} \|_{L^4(\chi)}^4 &= \inf_{c\in\C} \int_{\R^d} |u(\tau)-c|^4 \|F_\tau\|_{L^2}^4 \chi(\tau)\,\mbox{d}\tau\\
    &=  \inf_{c\in\D} \int_{\R^d} |u(\tau)-c|^4 w(\tau) \chi(\tau)\,\mbox{d}\tau\\
    &\le 4 \cdot \inf_{c\in\D} \int_{\R^d} |u(\tau)-c|^2 w(\tau) \chi(\tau)\,\mbox{d}\tau\\
    &= 4 \cdot \inf_{c\in\C} \|u-c\|_{L^2(w\chi)}^2.
\end{align*}
As $\chi\le 1$ we have that $\|\cdot\|_{L^4(\chi)}\le \|\cdot\|_{L^4}$.
Let us define two quantities by
$$
(I)\coloneqq \|\tau\mapsto \|H_\tau-u(\tau)F_\tau\|_{L^2}\|_{L^4}^4
\quad \text{and}\quad 
(II)\coloneqq \inf_{c\in\C} \|u-c\|_{L^2(w\chi)}^2.
$$
In order to arrive at the desired conclusion we need to show that 
\begin{equation}\label{eq:targetest}
    (I), (II) \le C (1+C_P(w\chi,w)) \|H\otimes\overline{H}-F\otimes\overline{F}\|_{L^2(\Gamma)}^2.
\end{equation}

We start with $(II)$. Applying the modified Poincar\'e inequality implies together with Theorem \ref{thm:poincmod} that 
$$
(II) \le C (1+C_P(w\chi,w))\|\delta_1[u]\|_{L^2(w)}^2.
$$
We introduce the short-hand notation
$$
\Lambda(\tau,y):= \|H_\tau-u(y)F_\tau\|_{L^2}^2\cdot \|F_\tau\|_{L^2}^2
$$
Suppose $y,\tau\in \R^d$ are such that $|y-\tau|\le 1$.
As $\gamma$ is a TSW, we get that there exists a function $\mu$ which is locally bounded such that 
$$
\|F_\tau\|_{L^2}^2 = \int_{\R^d} |F(x)|^2 \gamma(\tau-x)\,\mbox{d}x
\le \mu(\tau-y) \int_{\R^d} |F(x)|^2 \gamma(y-x)\,\mbox{d}y
= \mu(\tau-y)\|F_y\|_{L^2}^2.
$$
The same argument shows that $\|H_\tau-u(y)F_\tau\|_{L^2}^2\le \mu(\tau-y) \|H_y-u(y)F_y\|_{L^2}^2$. Since $\mu$ is locally bounded we get for all $y,\tau$ with $|y-\tau|\le 1$ that 
$$
\Lambda(\tau,y)\le C \Lambda(y,y).
$$
With this, 
\begin{align*}
    \delta_1[u](\tau)^2\cdot w(\tau) &= \int_{B_1(\tau)} |u(y)-u(\tau)|^2 \,\mbox{d}y\cdot \|F_\tau\|_{L^2}^4\\
    &= \int_{B_1(\tau)} \int_{\R^d} |u(y)F_\tau(z)-u(\tau)F_\tau(z)|^2\,\mbox{d}z\,\mbox{d}y\,\cdot \|F_\tau\|_{L^2}^2\\
    &\le 4 \left( \int_{B_1(\tau)}\Lambda(\tau,y) + \Lambda(\tau,\tau) \,\mbox{d}y\right)\\
    &\le C \left( \int_{B_1(\tau)} \Lambda(y,y)\,\mbox{d}y + \Lambda(\tau,\tau)\right).
\end{align*}
It follows from Lemma \ref{lem:hilbert2} that 
$$
\Lambda(y,y)\le 2 \|H_y\otimes\overline{H_y}-F_y \otimes\overline{F_y}\|_{L^2}^2,
$$
and likewise for $\Lambda(\tau,\tau)$.
Thus, we further have that
    \begin{equation}
        \delta_1[u](\tau)^2 \cdot \omega(\tau) 
        \le C \left(\int_{B_1(\tau)}\|H_y\otimes\overline{H_y}-F_y\otimes\overline{F_y}\|_{L^2}^2 \,\mbox{d}y + \|H_\tau\otimes\overline{H_\tau}-F_\tau\otimes\overline{F_\tau}\|_{L^2}^2 \right).
    \end{equation}

    Integrating both sides w.r.t. $\tau\in\R^d$, changing order of integration (and absorbing the constants) further gives
    \begin{multline*}
        \|\delta_1[u]\|_{L^2(w)}^2 \le C \iint_{\R^d\times\R^d} |F(z)\overline{F(z')}-H(z)\overline{H(z')}|^2 \\
        \times 
        \left\{ \int_{\R^d} \left[ \int_{B_1(\tau)} \gamma(y-z)\gamma(y-z')\,\mbox{d}y+  \gamma(\tau-z)\gamma(\tau-z')\right]\,\mbox{d}\tau\right\}
       \mbox{d}z\,\mbox{d}z'
    \end{multline*}
    The second term inside $\{\ldots\}$ 
    produces $\Gamma(z,z')$.
  Furthermore, a change of variable reveals
    \begin{equation*}
        \int_{\R^d} \int_{B_1(\tau)} \gamma(y-z)\gamma(y-z')\,\mbox{d}y\,\mbox{d}\tau
        = \int_{\R^d}\int_{B_1(y)}  \gamma(y-z)\gamma(y-z')\,\mbox{d}\tau \,\mbox{d}y\\
        = |B_1(0)|\cdot \Gamma(z,z').
    \end{equation*}
  This implies \eqref{eq:targetest} for $(II)$.\\

    The first term is significantly easier to deal with. Using Lemma \ref{lem:hilbert2} once more gives
    \begin{align*}
        (I) &= \int_{\R^d} \|H_\tau-u(\tau)F_\tau\|_{L^2}^4\,\mbox{d}\tau\\
        &\le \int_{\R^d} \|H_\tau-u(\tau)F_\tau\|_{L^2}^2 \cdot 2(\|H_\tau\|_{L^2}^2+\|F_\tau\|_{L^2}^2)\,\mbox{d}\tau\\
        &\le 4 \int_{\R^d} \|H_\tau\otimes \overline{H_\tau}-F_\tau\otimes \overline{F_\tau}\|_{L^2}^2\,\mbox{d}\tau.
    \end{align*}
    From this point one can argue in the same way as above to arrive at \eqref{eq:targetest} for $(I)$.
\end{proof}

We require a couple of auxiliary results.
First we compare two ways of how to measure phaseless differences.
\begin{lemma}\label{lem:relphaselessdiff}
    There exists a finite constant $C>0$ (depending on $d$ only) such that
    \begin{equation}
        \forall F,H\in L^4(\R^d): \quad \||H|-|F|\|_{\mathcal{L}}\le C \||H|^2-|F|^2\|_{L^2}^{1/2}.
    \end{equation}
\end{lemma}
\begin{proof}
    An application of Jensen's inequality and the 
    elementary inequality $(a-b)^2\le |a^2-b^2|$ for $a,b\ge0$,
    allows us to bound
    \begin{align*}
        \||H|-|F|\|_{\mathcal{L}}^4 
        &= \int_{\R^d} \left( \int_{B_1(y)} (|H(x)|-|F(x)|)^2\,\mbox{d}x \right)^2\,\mbox{d}y \\
        &\le C \int_{\R^d} \int_{B_1(y)} (|H(x)|-|F(x)|)^4\,\mbox{d}x\,\mbox{d}y\\
        &= C |B_1(0)| \cdot  \int_{\R^d} (|H(x)|-|F(x)|)^4\,\mbox{d}x\\
        &\le C |B_1(0)| \cdot  \int_{\R^d} (|H(x)|^2-|F(x)|^2)^2\,\mbox{d}x\\
        &= C |B_1(0)| \||H|^2-|F|^2\|_{L^2}^2,
    \end{align*}
    which implies the claim.
\end{proof}
The next result provides us with some flexibility when it comes to the constraint on the constant in the distance.
  We point out that a similar argument has been used in one of our previous articles \cite{grohs:gabormultivariate} and that it could be stated much more generally. 
\begin{lemma}\label{lem:replaceconstraint}
Let $\chi$ be a bounded weight on $\R^d$ and let $F,H\in L^4(\R^d)$. Then it holds that
$$
\inf_{\lambda\in \mathbb{T}} \|H-\lambda F\|_{\mathcal{L}(\chi)} \le 
2 \inf_{c\in\C} \|H-cF\|_{\mathcal{L}(\chi)} + \||H|-|F|\|_{\mathcal{L}(\chi)}.
$$
\end{lemma}
\begin{proof}
  Given $\varepsilon>0$ arbitrary, let $c_\varepsilon\in \C$ be such that 
  \begin{equation}\label{eq:estceps}
  \|H-c_\varepsilon F\|_{\mathcal{L}(\chi)} \le \inf_{c\in\C} \|H-cF\|_{\mathcal{L}(\chi)} +\varepsilon.
  \end{equation}
  By continuity we may assume w.l.o.g. that $c_\varepsilon\neq 0$.
  For every $\varepsilon>0$,
  \begin{equation*}
      \inf_{\lambda\in\mathbb{T}} \|H-\lambda F\|_{\mathcal{L}(\chi)} 
      \le \left\|H- \frac{c_\varepsilon}{|c_\varepsilon|}F \right\|_{\mathcal{L}(\chi)}
      \le \|H-c_\varepsilon F\|_{\mathcal{L}(\chi)} + \left\|c_\varepsilon F-\frac{c_\varepsilon}{|c_\varepsilon|}F \right\|_{\mathcal{L}(\chi)}.
  \end{equation*}
  While the first term on the right hand side can be bounded with \eqref{eq:estceps}, we get for the second one that 
  \begin{align*}
      \left\|c_\varepsilon F-\frac{c_\varepsilon}{|c_\varepsilon|}F \right\|_{\mathcal{L}(\chi)} &= \left\| |c_\varepsilon F|-|F| \right\|_{\mathcal{L}(\chi)}\\
      &\le \||c_\varepsilon F|-|H|\|_{\mathcal{L}(\chi)} + \||F|-|H|\|_{\mathcal{L}(\chi)} \\
      &\le \|c_\varepsilon F-H\|_{\mathcal{L}(\chi)} + \||F|-|H|\|_{\mathcal{L}(\chi)}.
  \end{align*}
  The statement now follows by sending $\varepsilon\to 0$.
\end{proof}

The proof of the first main result is now very short.
\begin{proof}[Proof of Theorem \ref{thm:abstractstability}]
    We simply have to combine the results of this section:
    \begin{align*}
        \inf_{\lambda\in\mathbb{T}} \|H-\lambda F\|_{\mathcal{L}(\chi)} 
        & \stackrel{\text{Lemma \ref{lem:replaceconstraint}}}{\le}
        2\inf_{c\in\C} \|H-c F\|_{\mathcal{L}(\chi)} + \||H|-|W|\|_{\mathcal{L}(\chi)}\\
        &\quad \stackrel{\chi\le 1}{\le} \quad 
        2\inf_{c\in\C} \|H-c F\|_{\mathcal{L}(\chi)} + \||H|-|W|\|_{\mathcal{L}}\\
        &\stackrel{\text{Lemma \ref{lem:relphaselessdiff}}}{\le} 
         2\inf_{c\in\C} \|H-c F\|_{\mathcal{L}(\chi)} + C \||H|^2-|F|^2\|_{L^2}^{1/2}
    \end{align*}
    Applying Proposition \ref{prop:abstractstability} and using the assumption that $F,H\in V$ and $V$ does $\Gamma$-stable lifted phase retrieval we get that 
    \begin{align*}
        \inf_{c\in\C} \|H-c F\|_{\mathcal{L}(\chi)}^4 &\le C'(1+C_P(w\chi,w)) \|H\otimes\overline{H}-F\otimes\overline{F}\|_{L^2(\Gamma)}^2\\
        &\le C''(1+C_P(w\chi,w)) \||H|^2-|F|^2\|_{L^2}^2.
    \end{align*}
    This finishes the proof.
\end{proof}

\subsection{STFT stability}\label{sec:stftstability}
The proof of the result establishing the link between local Lipschitz stability and Poincare\'e constant is very short as well as all the work has already been done.
\begin{proof}[Proof of Theorem \ref{thm:stabilitystft}]
    According to Proposition \ref{prop:rhostableprstft}
    we have that $V=\V_g(L^2(\R))$ does $\Gamma$-stable lifted phase retrieval.
    The statement now follows directly from the general result, Theorem \ref{thm:abstractstability}.
\end{proof}

\subsection{Finite Poincar\'e constants}\label{sec:finitepoinc}
This paragraph is dedicated to the proofs of Theorem \ref{thm:finitepoinc1} and Theorem \ref{thm:finitepoinc2}
\begin{proof}[Proof of Theorem \ref{thm:finitepoinc1}]
    Let us denote by $LC(\R^d)$ the family of all log-concave, integrable functions on $\R^d$.
    We will make use of the fact that  $LC(\R^d)$  is closed under convolution \cite{prekopa:logconcave}.
    We claim that $|\V_g g|^2 \in LC(\R^2)$ if $g(t)=\exp(t-e^t)$. Once we have verified this, we can argue 
    $$
    |\V_g g|^2\in LC(\R^2) \quad 
    \Rightarrow 
    \quad 
    |\V_g g|^2\ast \gamma \in LC(\R^2)
    \quad \Rightarrow \quad 
     (|\V_g g|^2 \ast \gamma)^2 \in LC(\R^2),
    $$
    and Bobkov's Theorem \ref{thm:bobkovlogconc} implies the statement.\\

    Recall that (cf. proof of Lemma \ref{lem:expadmpairs}) 
    \begin{equation}\label{eq:modsqVgg}
        |\V_g g(x,\xi)|^2 = \frac1{16} \sech\left(\frac{x}2\right)^4 |\Gamma(2-2\pi i \xi)|^2. 
    \end{equation}
    Thus, it suffices to show that both
    \begin{equation}
        \phi_1(x)=\sech(x),\quad \text{and}\quad \phi_2(\xi)= |\Gamma(2-i\xi)|^2=\frac{\pi\xi}{\sinh(\pi \xi)} (1+\xi^2) 
    \end{equation}
    are log-concave in the 1d sense.
    Note that if $\phi:\R\to \R_+$ is a smooth function then 
    $$
    (\ln \phi)''\ge 0 \quad \Leftrightarrow \quad  (\phi')^2-\phi'' \phi\ge 0.
    $$

    With this, since
    \begin{equation*}
        (\phi_1')^2-\phi_1''\phi = \tanh^2 \sech^2 - \sech^2 \left(\tanh^2-\sech^2 \right) = \sech^4>0.
    \end{equation*}
    we get that $\phi_1$ is log-concave.\\
    
    Proving log-concavity for the second function is a bit more involved. Note that by continuity and symmetry it suffices to show that $(\ln\phi_2)''(\xi)\ge 0$ only for $\xi>0$.
    For $\xi>0$ we have that 
    \begin{align*}
        (\ln \phi_2)''(\xi) &= \left( \ln\pi +\ln\xi - \ln(\sinh(\pi\xi)) + \ln(1+\xi^2) \right)''\\
        &= -\xi^{-2} + \pi^2 \csch^2(\pi \xi) - \frac{2(\xi^2-1)}{(1+\xi^2)^{2}} 
    \end{align*}
    Next we multiply the above expression by the positive function 
    $$
    (\star)=\sinh^2(\pi\xi) \xi^2 (1+\xi^2)^2
    $$
    and obtain
    \begin{align*}
        (\ln\phi_2)''\cdot (\star) &= -\sinh^2(\pi\xi) (1+\xi^2)^2 + \pi^2 \xi^2(1+\xi^2)^2 -2 \sinh^2(\pi\xi) \xi^2(\xi^2-1)\\
        &= -\sinh^2(\pi\xi) \left[(1+\xi^2)^2+2\xi^2(\xi^2-1) \right] + \pi^2\xi^2(1+\xi^2)^2\\
        &= -\sinh^2(\pi\xi)\left[3\xi^4+1 \right] + \pi^2\xi^2(1+\xi^2)^2\\
        &\le - \sinh^2(\pi\xi) + \pi^2\xi^2(1+\xi^2)^2
    \end{align*}
    By substituting $\pi\xi=v$, we see that it is sufficient to show that 
    \begin{equation}\label{eq:claimsinhsq}
    \forall v>0:\quad \sinh^2(v) \ge v^2\left( 1 + \frac{v^2}{\pi^2}\right)^2 = v^2 + \frac2{\pi^2}v^4 + \frac1{\pi^4} v^6
    \end{equation}
    The Taylor expansion of the function on the left hand side is 
    \begin{equation*}
        \sinh^2(v) = \frac12 \cosh(2v) - \frac12
        = \frac12 \left(\sum_{\ell=0}^\infty \frac{(2v)^{2\ell}}{(2\ell)!} -1\right)
        = 
         \sum_{\ell=1}^\infty \frac{2^{2\ell-1}}{(2\ell)!}v^{2\ell}
    \end{equation*}
    Note that all the coefficients are positive, and that 
    $$
    \frac{2^{2-1}}{2!}=1,\quad \frac{2^{4-1}}{4!}> \frac2{\pi^2},\quad \text{and}\quad \frac{2^{6-1}}{6!} > \frac1{\pi^4},
    $$
    which proves \eqref{eq:claimsinhsq}, and we are done.
\end{proof}
We continue with the proof of the second statement.
\begin{proof}[Proof of Theorem \ref{thm:finitepoinc2}]
    First we shall show that 
    $$
    \sqrt{w} = (|\V_g g|^2\ast \gamma)  \asymp \phi\otimes\psi
    $$
    with $\phi$ log-concave and integrable and $\psi=\frac1{1+\cdot^2}$. 
    Recall (cf. proof of Lemma \ref{lem:expadmpairs}) that 
    $$
    |\V_g g(x,\xi)|^2 = \frac{e^{-2|x|}}{4(1+\pi^2\xi^2)}.
    $$
    Thus, 
    $$
    \sqrt{w}= \frac14 (e^{-2|\cdot|}\ast e^{-a|\cdot|})\otimes \left( \frac1{1+\pi^2\cdot^2} \ast e^{-b|\cdot|} \right)
    $$
    The first term in the tensor product is a convolution of two log-concave functions, and hence log-concave itself.
    Therefore, it remains to show that 
    \begin{equation}\label{eq:equivconvprod}
        \frac1{1+\pi^2\cdot^2} \ast e^{-b|\cdot|} \asymp  \psi.
    \end{equation}
    We will use that $\frac1{1+c\cdot^2}\asymp \psi$ for any $c>0$ (with the implicit constant depending on $c$ only). 
    Moreover, since $e^{-b|\cdot|}\lesssim \psi$ we get that 
    $$
    \frac1{1+\pi^2\cdot^2} \ast e^{-b|\cdot|} \lesssim \psi \ast\psi = \frac{2\pi}{4+\cdot^2} \lesssim \psi.
    $$
    For the other direction, we have that 
    \begin{equation}
        \left(\frac1{1+\pi^2\cdot^2} \ast e^{-b|\cdot|}\right)(x) \gtrsim (\psi\ast \mathds{1}_{(-\frac12,\frac12)})(x)
        = \int\limits_{x-\frac12}^{x+\frac12} \psi(t)\,\mbox{d}t.
    \end{equation}
    Note that $\psi\ge\frac12$ on $[-1,1]$ and that $\psi$ is monotonically decreasing on $(0,\infty)$.
    By symmetry, we further have that 
    $$
    \int\limits_{x-\frac12}^{x+\frac12} \psi(t)\,\mbox{d}t \ge 
    \begin{cases}
    \frac12,\quad &|x|\le \frac12,\\
    \psi\left(x+\frac12 \right), \quad &|x|>\frac12,
    \end{cases}
    $$
    which is $\asymp\psi$, and we proved \eqref{eq:equivconvprod}.\\
    
    Remember that we have to show that $C_P(w\chi,w)<\infty$ with $\chi(x,\xi)=\psi(\xi)$.
    Clearly, replacing the weight in either slot by an equivalent one does not change the property of yielding a finite Poincar\'e constant (however, in general the value will change). This means, it suffices to show that 
    $$
    C_P(\phi^2 \otimes \psi^3, \phi^2\otimes\psi^2)< \infty.
    $$
    According to the tensorisation property, Lemma \ref{lem:tensorpoincare}, it is enough to establish the one-dimensional Poincar\'e inequalities:
    a) $C_P(\phi^2)<\infty$ and b) $C_P(\psi^3,\psi^2)<\infty$.
    But, since $\phi$ is log-concave, so is $\phi^2$ and a) follows from Theorem \ref{thm:bobkovlogconc}.
    On the other hand, statement b) is a direct consequence of Theorem \ref{bobkov:cauchy} with $\beta=2$.
\end{proof}

\section*{Acknowledgements}
The author was supported by the Erwin–Schr{\"o}dinger Program (J-4523) of
the Austrian Science Fund (FWF).

\bibliographystyle{plain}
\bibliography{bibliography.bib}

\begin{thebibliography}{10}

\bibitem{alaifari:banach}
Rima Alaifari and Philipp Grohs.
\newblock Phase retrieval in the general setting of continuous frames for {B}anach spaces.
\newblock {\em SIAM J. Math. Anal.}, 49(3):1895--1911, 2017.

\bibitem{lecturenotes:concofmeas}
Nathana\"{e}l Berestycki and Richard Nickl.
\newblock Lecture notes: Concentration of measure, December 2009.
\newblock University of Cambridge.

\bibitem{bobkov:isoplogconcave}
S.~G. Bobkov.
\newblock Isoperimetric and analytic inequalities for log-concave probability measures.
\newblock {\em Ann. Probab.}, 27(4):1903--1921, 1999.

\bibitem{bobkov:cauchy}
Sergey~G. Bobkov and Michel Ledoux.
\newblock {Weighted Poincaré-type inequalities for Cauchy and other convex measures}.
\newblock {\em The Annals of Probability}, 37(2):403 -- 427, 2009.

\bibitem{bonnefont:lecturenotes}
Michel Bonnefont.
\newblock Poincaré inequality with explicit constant in dimension $d \ge 1$, July 2022.
\newblock Global Sensitivity Analysis and Poincaré inequalities (Summer school, Toulouse).

\bibitem{cahill:infdimhilbert}
Jameson Cahill, Peter~G. Casazza, and Ingrid Daubechies.
\newblock Phase retrieval in infinite-dimensional {H}ilbert spaces.
\newblock {\em Trans. Amer. Math. Soc. Ser. B}, 3:63--76, 2016.

\bibitem{calderbank2022stable}
Robert Calderbank, Ingrid Daubechies, Daniel Freeman, and Nikki Freeman.
\newblock Stable phase retrieval for infinite dimensional subspaces of $l_2(\mathbb{R})$, 2022.

\bibitem{Cheeger:lowerbound}
Jeff Cheeger.
\newblock {\em A Lower Bound for the Smallest Eigenvalue of the Laplacian}, pages 195--200.
\newblock Princeton University Press, Princeton, 1971.

\bibitem{christ2023examples}
Michael Christ, Ben Pineau, and Mitchell~A. Taylor.
\newblock Examples of h\"older-stable phase retrieval.
\newblock {\em Math. Res. Lett.}, 2023.
\newblock (to appear).

\bibitem{CORBETT200653}
J.V. Corbett.
\newblock The pauli problem, state reconstruction and quantum-real numbers.
\newblock {\em Reports on Mathematical Physics}, 57(1):53--68, 2006.

\bibitem{dainty:astronomy}
J.~Dainty and James Fienup.
\newblock Phase retrieval and image reconstruction for astronomy.
\newblock {\em Image Recovery: Theory Appl}, 13, 01 1987.

\bibitem{evans:pdes}
Lawrence~C. Evans.
\newblock {\em Partial differential equations}, volume~19 of {\em Graduate Studies in Mathematics}.
\newblock American Mathematical Society, Providence, RI, 1998.

\bibitem{freeman:stablefctspace}
D.~{Freeman}, T.~{Oikhberg}, B.~{Pineau}, and M.~A. {Taylor}.
\newblock {Stable phase retrieval in function spaces}.
\newblock {\em Math. Ann.}, 2022.
\newblock (to appear).

\bibitem{grochenig:foundations}
Karlheinz Gr\"{o}chenig.
\newblock {\em Foundations of time-frequency analysis}.
\newblock Applied and Numerical Harmonic Analysis. Birkh\"{a}user Boston, Inc., Boston, MA, 2001.

\bibitem{groechenig:mystery}
Karlheinz Gr\"{o}chenig.
\newblock The mystery of {G}abor frames.
\newblock {\em J. Fourier Anal. Appl.}, 20(4):865--895, 2014.

\bibitem{grochenig:zeros}
Karlheinz Gr\"{o}chenig, Philippe Jaming, and Eugenia Malinnikova.
\newblock Zeros of the {W}igner distribution and the short-time {F}ourier transform.
\newblock {\em Rev. Mat. Complut.}, 33(3):723--744, 2020.

\bibitem{liehr:founddiscbarr}
Philipp Grohs and Lukas Liehr.
\newblock On foundational discretization barriers in {STFT} phase retrieval.
\newblock {\em J. Fourier Anal. Appl.}, 28(2):Paper No. 39, 21, 2022.

\bibitem{grohs:gaborspectral}
Philipp Grohs and Martin Rathmair.
\newblock Stable {G}abor phase retrieval and spectral clustering.
\newblock {\em Comm. Pure Appl. Math.}, 72(5):981--1043, 2019.

\bibitem{grohs:gabormultivariate}
Philipp Grohs and Martin Rathmair.
\newblock Stable {G}abor phase retrieval for multivariate functions.
\newblock {\em J. Eur. Math. Soc. (JEMS)}, 24(5):1593--1615, 2022.

\bibitem{miao:beyondcrystallography}
Jianwei Miao, Tetsuya Ishikawa, Ian~K. Robinson, and Margaret~M. Murnane.
\newblock Beyond crystallography: Diffractive imaging using coherent x-ray light sources.
\newblock {\em Science}, 348(6234):530--535, 2015.

\bibitem{prekopa:logconcave}
Andr\'as Pr\'ekopa.
\newblock On logarithmic concave measures and functions.
\newblock {\em Acta Sci. Math. (Szeged)}, 34:335--343, 1973.

\bibitem{waldspurger:waveletpr}
Irène Waldspurger.
\newblock Phase retrieval for wavelet transforms.
\newblock {\em IEEE Transactions on Information Theory}, 63(5):2993--3009, 2017.

\end{thebibliography}

\end{document}